\documentclass{ws-aejm}

\usepackage{xcolor}
\usepackage[verbose]{hyperref}
\hypersetup{colorlinks=false,allbordercolors=blue,pdfborderstyle={/S/U/W 1}}

\newcommand{\fix}{\mathrm{fix\, }}

\newcommand{\im}{\mathrm{im\, }}
\newcommand{\rank}{\mathrm{rank\, }}

\begin{document}
	
\setcounter{page}{1}
	
\markboth{E. Korkmaz and H. Ayık }{Combinatorial results for zero-divisors regarding right zero elements of order-preserving transformations}
	
%
\catchline{}{}{}{}{}
%
	
\title{Combinatorial results for zero-divisors regarding right zero elements of order-preserving transformations}

\author{Emrah Korkmaz}
	
\address{Department of Mathematics, Çukurova University\\
		Adana 01330, Türkiye\\
		\email{ekorkmaz@cu.edu.tr}}
	
\author{Hayrullah Ayık}
	
\address{Department of Mathematics, Çukurova University\\
		Adana 01330, Türkiye\\
		\email{hayik@cu.edu.tr}}
	
\maketitle
	
\begin{history}
	\comby{(Handling Editor)}
	\received{(Day Month Year)}
	\revised{(Day Month Year)}
	\accepted{(Day Month Year)}
	\published{(Day Month Year)}
\end{history}
	
\begin{abstract}
	For any positive integer $n$, let $\mathcal{O}_{n}$  be the semigroup of all order-preserving full transformations on $X_{n}=\{1<\cdots <n\}$. For any $1\leq k\leq n$, let  $\pi_{k}\in \mathcal{O}_{n}$ be the constant map defined by  $x\pi_{k}=k$ for all $x\in X_{n}$. In this paper, we introduce and study the sets of left, right, and two-sided zero-divisors of $\pi_{k}$:
	\begin{eqnarray*}
		\mathsf{L}_{k} &=& \{ \alpha\in \mathcal{O}_{n}:\alpha\beta=\pi_{k} \mbox{ for some }\beta\in \mathcal{O}_{n} \setminus\{\pi_{k}\} \}, \\
		\mathsf{R}_{k} &=& \{ \alpha\in \mathcal{O}_{n}:\gamma\alpha=\pi_{k} \mbox{ for some }\ \gamma\in \mathcal{O}_{n}\setminus\{\pi_{k}\} \}, \ \mbox{and} \ \mathsf{Z}_{k}=\mathsf{L}_{k}\cap \mathsf{R}_{k}. 
	\end{eqnarray*}
We determine the structures and cardinalities of $\mathsf{L}_{k}$, $\mathsf{R}_{k}$ and $\mathsf{Z}_{k}$ for each $1\leq k\leq n$. Furthermore, we compute the ranks of $\mathsf{R}_{1}$,\, $\mathsf{R}_{n}$,\, $\mathsf{Z}_{1}$,\, $\mathsf{Z}_{n}$ and $\mathsf{L}_{k}$ for each $1\leq k\leq n$, because these are significant subsemigroups of $\mathcal{O}_{n}$.
\end{abstract}
	
	
	
\section{Introduction}\label{sec1}

For $n\in \mathbb{N}$, let $\mathcal{T}_{n}$ denote the full transformation semigroup on the chain $X_{n}=\{1, \ldots ,n\}$ under its natural order. A transformation $\alpha\in \mathcal{T}_{n}$ is called \emph{order-preserving} if $x<y$ implies $x\alpha \leq y\alpha$ for all $x,y\in X_{n}$, and is called \emph{order-decreasing} (\emph{order-increasing}) if $x\alpha \leq x$ ($x\leq x\alpha$) for all $x\in X_{n}$. We denote the subsemigroup of $\mathcal{T}_{n}$ consisting of all order-preserving full transformations by $\mathcal{O}_{n}$. The \emph{fix} and \emph{image} sets of a transformation $\alpha\in \mathcal{T}_{n}$ are defined by $\fix(\alpha )=\{ x\in X_{n}: x\alpha=x\}$ and  $\im(\alpha)=\{x\alpha :x\in X_{n}\}$, respectively.
	
For a non-empty subset $A$ of a semigroup $S$, the smallest subsemigroup of $S$ containing $A$ is denoted by $\langle A\,\rangle$. For a subset $A$ of $S$, if $S=\langle A \rangle$, then $A$ is called a \emph{generating set} of $S$. The \emph{rank} of a semigroup $S$ is defined by $\rank(S)=\min \{\, \lvert A\rvert : S=\langle A\, \rangle \}$. A generating set of $S$ with size $\rank(S)$ is called a \emph{minimal generating set} of $S$. An element $s\in S$ is called \emph{undecomposable} if there are no $a,b\in S\setminus \{s\}$ such that $s=ab$. It is clear that every generating set of $S$ contains all undecomposable elements of $S$. For additional information on semigroup theory, the reader is advised to consult \cite{Howie:1995}.

Semigroups of order-preserving transformations have been the focus of sustained interest for over six decades. The earliest known studies of $\mathcal{O}_{n}$ date back to 1962, when A\u\i zen\v stat  \cite{Aizenstat:1962,Aizenstat1:1962} provided a presentation of $\mathcal{O}_n$ and described its congruences.  In 1971, Howie \cite{Howie:1971} computed both the cardinality  and the number of idempotents of $\mathcal{O}_{n}$. This was followed by a joint work with Gomes \cite{Gomes&Howie:1992}, where the rank and the idempotent rank of $\mathcal{O}_{n}$  were determined. Since then, a wide range of research has investigated the algebraic and combinatorial aspects of $\mathcal{O}_{n}$ and its subsemigroups. For more recent developments related to this paper, the reader is referred \cite{Korkmaz:2025,Koppitz&Worawiest:2022,Fernandes:1997,Fernandes&Volkov:1997} for further research within the scope of this study. 
	
In this paper, we consider the subsemigroups of $\mathcal{O}_{n}$ arising from its right zero elements. For each $1\leq k\leq n$, let 
$\pi_{k}=\left(\begin{smallmatrix}
	1 & 2 & \cdots & n \\
	k & k & \cdots & k
\end{smallmatrix}\right)$, 
one of the right zero elements of $\mathcal{O}_{n}$. For each $1\leq k\leq n$, let $\mathcal{O}_{n}^{k} =\mathcal{O}_{n} \setminus\{\pi_{k}\}$, and then, we define the following sets:
\begin{eqnarray*}
	\mathsf{L}_{k} &=& \{ \alpha \in \mathcal{O}_{n} : \alpha \beta= \pi_{k} \mbox{ for some } \beta\in \mathcal{O}_{n}^{k} \}, \\
	\mathsf{R}_{k} &=& \{ \alpha \in \mathcal{O}_{n} : \gamma \alpha =\pi_{k} \mbox{ for some } \gamma \in \mathcal{O}_{n}^{k} \}\, \mbox{ and}\\
	\mathsf{Z}_{k} &=&\mathsf{L}_{k}\cap \mathsf{R}_{k}= \{ \alpha \in \mathcal{O}_{n} : \alpha \beta =\pi_{k}= \gamma \alpha \mbox{ for some }\beta, \gamma \in \mathcal{O}_{n}^{k} \}.
	\end{eqnarray*}
For each $1\leq k\leq n$,  the sets $\mathsf{L}_{k}$, $\mathsf{R}_{k}$ and $\mathsf{Z}_{k}$ are called \emph{the set of left zero-divisors}, \emph{the set of right zero-divisors} and \emph{the set of two-sided zero-divisors} of $\mathcal{O}_{n}$ with respect to $\pi_{k}$, respectively. 
	
Let $\mathcal{IO}_{n}$ be the subsemigroup of $\mathcal{O}_{n}$ consisting of all transformations in $\mathcal{O}_{n}$ whose images are convex subsets (intervals) of $X_{n}$. Fernandes and Paulista showed that $\mathcal{IO}_{n}$ coincides with the subsemigroup of all \emph{weak endomorphisms} of a directed path with $n$ vertices in \cite{Fernandes&Paulista:2023}. Moeover, the authors determined the cardinality and rank of $\mathcal{IO}_{n}$.
Thereafter, in \cite{Fernandes:2024}, Fernandes gave a presentation for the subsemigroup $\mathcal{IO}_{n}$. Such results underline the pivotal role of $\mathcal{IO}_n$ for examining $L_k$ ($1 \leq k \leq n$).
	
The constant maps $\pi_k$ are idempotent right zero elements of  $\mathcal{O}_n$. Despite their natural algebraic role, the subsemigroups they generate remain unexplored. Notably, their interaction with the left zero-divisors $\mathsf{L}_k$ and $\mathcal{IO}_n$ suggests structural significance from both algebraic, combinatorial and graph-theoretic perspectives.
	
A motivating precedent comes from the full transformation semigroup $\mathcal{T}_n$, where right zero elements like $\pi_k$ have been studied via graph-theoretic constructions in \cite{Korkmaz:2025}. The collection of studies on zero-divisor graphs in 
 \cite{Anderson&Livingston:1999,Korkmaz:2002,DeMeyer&DeMeyer:2005,Redmond:2002,Toker:2021,Toker:2021a}  suggest that analogous techniques may yield fruitful insights within $\mathcal{O}_n$ as well.

Given that $\pi_k$ behaves similarly in $\mathcal{O}_n$ and $\mathcal{T}_n$, and in light of the established graph-theoretic frameworks for $\mathcal{T}_n$ and $\mathcal{IO}_n$, the results presented in this paper lay a promising foundation for further research.

\section{Zero Divisors of $\mathcal{O}_{n}$}\label{sec2}

For any $n,r\in \mathbb{N}$, the number of ordered non-negative integers solutions of the equation $x_{1}+x_{2}+\cdots+x_{r}=n$ is known to be $\binom{n+r-1}{r-1}$. As noted in \cite{Howie:1971}, the cardinality of $\mathcal{O}_{n}$ is equal to the number of non-negative integer solutions of the equation $x_{1}+x_{2}+\cdots+x_{n}=n$, that is $\lvert \mathcal{O}_{n} \rvert=\binom{2n-1}{n-1}$. By using a similar technique, we have the following result:

\smallskip

\begin{lemma}\label{l1}
	For $n\geq 2$,
\begin{enumerate}[$(i)$]
	\item $\mathsf{L}_{1}=\{ \alpha\in \mathcal{O}_{n} : n\notin \im(\alpha)\}$ and  $\lvert \mathsf{L}_{1} \rvert= \binom{2n-2}{n-2}$;
	\item $\mathsf{L}_{n}=\{ \alpha\in \mathcal{O}_{n} : 1\notin \im(\alpha)\}$ and  $\lvert \mathsf{L}_{n} \rvert= \binom{2n-2}{n-2}$; and
	\item  $\mathsf{L}_{k} =\mathsf{L}_{1} \cup \mathsf{L}_{n} =\{ \alpha\in \mathcal{O}_{n} : 1\notin \im(\alpha) \mbox{ or } n\notin \im(\alpha)\}$ and $\lvert \mathsf{L}_{k} \rvert= \binom{2n-2}{n-2} +\binom{2n-3}{n-2}$ for $n\geq 3$ and each $2\leq k\leq n-1$.
\end{enumerate}
\end{lemma}

\begin{proof} $(i)$ Let $\alpha \in \mathsf{L}_{1}$ and suppose that $\alpha\beta=\pi_{1}$ for some $\beta\in \mathcal{O}_{n}^{1}$. Assume that $n\in \im(\alpha)$. Since $1\leq x\beta \leq n\beta =1$, we have $x\beta=1$ for all $x\in  X_{n}$, and so $\beta=\pi_{1}$ which is a contradiction. Therefore, $n\notin \im(\alpha)$.
	
Conversely, let $\alpha \in \mathcal{O}_{n}$ with $n\notin \im(\alpha)$. If we consider the mapping 
$\beta= \left( \begin{smallmatrix}
	1 &\cdots & n-1 & n \\
	1 &\cdots &  1  & n 
\end{smallmatrix}\right),$
then it is clear that $\beta\in \mathcal{O}_{n}^{1}$ and $\alpha \beta= \pi_{1}$, and hence, $\alpha\in \mathsf{L}_{1}$.
	
Thus, the cardinality of $\mathsf{L}_{1}$ is equal to the number of ordered integer solutions of the equation $x_{1}+x_{2}+\cdots+ x_{n-1} =n$ with $x_{1},x_{2},\ldots, x_{n-1}\geq0$, i.e. $\lvert \mathsf{L}_{1} \rvert= \binom{2n-2}{n-2}$. 
	
$(ii)$ The proof is similar to the proof $(i)$.
	
$(iii)$ For a fixed $2\leq k\leq n-1$, let $\alpha\in \mathsf{L}_{k}$ and suppose that $\alpha \beta= \pi_{k}$ for some $\beta\in \mathcal{O}_{n}^{k}$. Assume that  $\{1,n\} \subseteq \im(\alpha)$. Then we must have $1\beta= n\beta =k$, and so $\beta= \pi_{k}$ which is a contradiction. Therefore, $1\notin \im(\alpha)$ or $n\notin \im(\alpha)$, and so $\alpha\in \mathsf{L}_{n}$ or $\alpha \in \mathsf{L}_{1}$, respectively.
	
Conversely, suppose $\alpha\in \mathsf{L}_{1}\cup \mathsf{L}_{n}$. If we consider 
$\beta_{k}= \left( \begin{smallmatrix}
	1 &\cdots & n-1 & n \\
	k &\cdots &  k  & n 
\end{smallmatrix}\right)$ and 
$\gamma_{k}= \left( \begin{smallmatrix}
	1 & 2 &\cdots & n \\
	1 & k &\cdots & k  
\end{smallmatrix}\right)$,
then it is clear that $\beta_{k}, \gamma_{k} \in \mathcal{O}_{n}^{k}$, and moreover, $\alpha \beta_{k} =\pi_{1}$ if  $\alpha\in \mathsf{L}_{1}$ or $\alpha \gamma_{k} =\pi_{1}$ if $\alpha\in \mathsf{L}_{n}$, and hence, $\alpha\in \mathsf{L}_{k}$.
	
Since the cardinality of $\mathsf{L}_{1} \cap \mathsf{L}_{n}= \{\alpha\in \mathcal{O}_{n} : 1\notin \im(\alpha) \mbox{ and } n\notin \im(\alpha) \}$ is equal to the number of ordered integer solutions of the equation  $x_{2}+x_{3} +\cdots + x_{n-1}=n$ with $x_{2},\ldots, x_{n-1}\geq0$, i.e. $\lvert \mathsf{L}_{1} \cap \mathsf{L}_{n} \rvert= \binom{2n-3}{n-3}$, it follows that
	$$\lvert \mathsf{L}_{k} \rvert= \lvert \mathsf{L}_{1} \rvert +\lvert \mathsf{L}_{n} \rvert- \lvert \mathsf{L}_{1} \cap \mathsf{L}_{n} \rvert= 2\binom{2n-2}{n-2} -\binom{2n-3}{n-3} =\binom{2n-2}{n-2} +\binom{2n-3}{n-2},$$
for each $2\leq k\leq n-1$. 
\end{proof} 

We now determine the structures and cardinalities of the set of right zero-divisors of $\mathcal{O}_{n}$ with respect to $\pi_{k}$ for each $1\leq k\leq n$.

\smallskip

\begin{lemma}\label{l2}
	Let $n\geq 2$ and $1\leq k\leq n$. Then we have
	$$\mathsf{R}_{k}=\{ \alpha\in \mathcal{O}_{n} : k\in \im(\alpha) \,\mbox{ and }\, k\alpha^{-1} \neq \{k\}\},$$ and moreover, $\lvert \mathsf{R}_{1}\rvert= \lvert \mathsf{R}_{n} \rvert= \binom{2n-3}{n-1}$ and $\lvert \mathsf{R}_{k} \rvert= \binom{2n-2}{n-1} -\binom{2k-3}{k-2} \binom{2n-2k-1}{n-k-1}$ for $2\leq k\leq n-1$.
\end{lemma}

\begin{proof} Let $1\leq k\leq n$ and let $\alpha\in \{\alpha \in \mathcal{O}_{n} : k\in \im(\alpha) \,\mbox{ and }\, k\alpha^{-1} \neq \{k\}\}$. Then there exists $i\in X_{n} \setminus \{k\}$ such that $i\alpha=k$, and so $\pi_{i} \alpha =\pi_{k}$. Since $\pi_{i} \in \mathcal{O}_{n}^{k}$, we have $\alpha\in \mathsf{R}_{k}$.
		
Conversely, let $\alpha\in \mathsf{R}_{k}$ and suppose that $\beta \alpha= \pi_{k}$ for some $\beta\in \mathcal{O}_{n}^{k}$. Since $\{k\}= \im(\pi_{k}) \subseteq \im(\alpha)$, we have $k\in \im(\alpha)$. In addition, since $k\alpha^{-1} =\{k\}$ implies $\beta= \pi_{k}$, which is a contradiction, we conclude that $k\alpha^{-1} \neq \{k\}$ for all $1\leq k\leq n$. 
	
Now we determine the cardinalities. For each $\alpha\in \mathsf{R}_{1}$, since $1\in \im(\alpha)$ and $1\alpha^{-1} \neq \{1\}$, it follows that $1\alpha= 2\alpha=1$, and so the cardinality of $\mathsf{R}_{1}$ is equal to the number of ordered integer solutions of the equation  $x_{1}+x_{2}+\cdots+x_{n}=n$ with $x_{1}\geq 2$ and $x_{2},\ldots, x_{n}\geq0$, i.e. $\lvert \mathsf{R}_{1}\rvert =\binom{2n-3}{n-1}$. Similarly, we have $\lvert \mathsf{R}_{n} \rvert= \binom{2n-3}{n-1}$.
	
For $n\geq 3$ and $2\leq k\leq n-1$, if we consider the disjoint sets
	$$A_{k}=\{ \alpha\in \mathcal{O}_{n} : k\alpha^{-1} =\{k\} \}\,\mbox{ and }\, B_{k}= \{ \alpha\in \mathcal{O}_{n} : k\not\in \im(\alpha)\},$$ 
then it is evident that $\mathsf{R}_{k}= \mathcal{O}_{n} \setminus (A_{k}\cup B_{k})$. Now we use a similar technique to the one introduced in \cite[Lemma 6]{Ayık&Koç:2011}. For each $\alpha\in A_{k}$, if we let 
\begin{eqnarray*}
	&&\alpha_{1}=\left(\begin{matrix}
		1  & 2       & \cdots & k-1 \\
			1\alpha  & 2\alpha & \cdots & (k-1)\alpha
	\end{matrix}\right) \, \mbox{ and}\\
	&&\alpha_{2}=\left(\begin{matrix}
		1              &           2       & \cdots & n-k \\
		(k+1)\alpha-k  & (k+2)\alpha-k & \cdots & n\alpha-k
	\end{matrix}\right), 
\end{eqnarray*}
then it is clear that $\alpha_{1}\in \mathcal{O}_{k-1}$ and $\alpha_{2}\in \mathcal{O}_{n-k}$. Moreover, if we define the mapping
\begin{eqnarray*}
	\varphi :A_{k}\rightarrow \mathcal{O}_{k-1} \times \mathcal{O}_{n-k}\, \mbox{ by }\, \alpha \varphi= (\alpha_{1},\alpha_{2})
\end{eqnarray*}
for all $\alpha\in A_{k}$, then it is a routine matter to check that $\varphi$	is a bijection, and so $\lvert A_{k}\rvert= \binom{2k-3}{k-2} \binom{2n-2k-1}{n-k-1}$. In addition, since $\binom{2n-1}{n-1}- \binom{2n-2}{n-2}= \binom{2n-2}{n-1}$ for all $n\geq 2$, and since the cardinality of $B_{k}$ is equal to the number of ordered integer solutions of the equation 
	$$x_{1}+\cdots +x_{k-1}+ x_{k+1}+\cdots +x_{n}=n\,\, \mbox{ with }\,\, x_{1},\ldots, x_{k-1}, x_{k+1} ,\ldots, x_{n}\geq 0,$$ 
i.e. $\lvert B_k\rvert= \binom{2n-2}{n-2}$, we have $\lvert \mathsf{R}_{k} \rvert =\binom{2n-2}{n-1} -\binom{2k-3}{k-2} \binom{2n-2k-1}{n-k-1}$, as claimed. 
\end{proof}

Finally we determine the cardinalities of the set of zero-divisors of $\mathcal{O}_{n}$ with respect to $\pi_{k}$ for each $1\leq k\leq n$. Since it is easy to check that $\mathsf{Z}_{1}= \mathsf{R}_{1}$ and $\mathsf{Z}_{n}= \mathsf{R}_{n}$ for $n=2$. we suppose $n\geq 3$ in the following lemma.

\smallskip

\begin{lemma}\label{l3}
	Let $n\geq 3$. Then we have $\lvert \mathsf{Z}_{1}\rvert=\lvert \mathsf{Z}_{n}\rvert=\binom{2n-4}{n-2}$ and for each $2\leq k\leq n-1$, we have $\lvert \mathsf{Z}_{k}\rvert=\lvert \mathsf{R}_{k}\rvert -\left( \binom{2n-4}{n-1} -\binom{2k-4}{k-2} \binom{2n-2k-2}{n-k-1} \right)$.
\end{lemma}

\begin{proof} Since $\mathsf{Z}_{1}=\{ \alpha\in \mathcal{O}_{n} : 1\alpha = 2\alpha =1\,\mbox{ and }\, n\notin \im(\alpha) \}$, one can easily determine as above that $\lvert \mathsf{Z}_{1} \rvert= \binom{2n-4}{n-2}$, and also $\lvert \mathsf{Z}_{n}\rvert=\binom{2n-4}{n-2}$. Moreover, it is easy to check that $\mathsf{Z}_{2}= \mathsf{R}_{2}$ for $n=3$. 
	
Suppose $n\geq 4$, and let $A=\{\alpha\in \mathsf{R}_{k}: \{1,n\} \subseteq \im(\alpha) \}$ for $2\leq k\leq n-1$. It immediately follows from Lemmas \ref{l1} and \ref{l2} that $\mathsf{Z}_{k} =\mathsf{R}_{k} \setminus A$. If we let 
\begin{eqnarray*}
	B&=&\{ \alpha\in \mathcal{O}_{n} :\{1,n\} \subseteq \im(\alpha) \mbox{ and } k\in \im(\alpha)\}\, \mbox{ and}\\
	C&=&\{ \alpha\in \mathcal{O}_{n} :\{1,n\} \subseteq \im(\alpha) \mbox{ and } k\alpha^{-1}=\{k\}\},
\end{eqnarray*}
then it is clear from Lemma \ref{l2} that $A=B\setminus C$. By considering the equation $x_{1}+x_{2} +\cdots+ x_{n}=n$ with $x_{2},\ldots, x_{k-1}, x_{k+1} ,\ldots, x_{n-1} \geq0$ and $x_{1} ,x_{k}, x_{n} \geq1$, we similarly have $\lvert B\rvert= \binom{2n-4}{n-1}$. Moreover, if we let 
	$$D=\{ \alpha\in \mathcal{O}_{k-1} : 1\alpha=1\}\, \mbox{ and }\, E=\{ \alpha\in \mathcal{O}_{n-k} : (n-k)\alpha=n-k\},$$ 
then for each $\alpha\in C$, with a similar technique to the one used in the proof of Lemma \ref{l2}, we define the bijection $\varphi :C \rightarrow D\times E$ by $\alpha \phi= (\alpha_{1},\alpha_{2})$ where
\begin{eqnarray*}
\alpha_{1}&=&\left(\begin{array}{cccc}
		1 &     2   & \cdots &     k-1    \\
		1 & 2\alpha & \cdots & (k-1)\alpha
\end{array}\right)\in D \mbox{ and}\\
\alpha_{2}&=&\left(\begin{array}{cccc}
		1      & \cdots &    n-k-1     & n-k \\
		(k+1)\alpha-k  & \cdots &(n-1)\alpha-k & n-k
\end{array}\right) \in E
\end{eqnarray*}
so that $\lvert C\rvert =\lvert D\rvert \, \lvert E\rvert= \binom{2k-4}{k-2} \binom{2n-2k-2}{n-k-1}$, and so the proof is now completed. 
\end{proof}

\section{Generating sets and ranks}\label{sec3}

In this section, since $\mathsf{R}_{k}$ and $\mathsf{Z}_{k}$ are subsemigroups of $\mathcal{O}_{n}$ if and only if $k=1$ or $k=n$, and since $\mathsf{L}_{k}$ is a subsemigroup of $\mathcal{O}_{n}$ for each $1\leq k\leq n$, we determine  some generating sets and the ranks of $\mathsf{R}_{1}$, $\mathsf{R}_{n}$, $\mathsf{Z}_{1}$, $\mathsf{Z}_{n}$ and $\mathsf{L}_{k}$ for all  $1\leq k\leq n$. Furthermore, to avoid stating the obvious, we suppose that $n\geq 3$ for the rest of the paper. 

A non-empty subset $A$ of $X_{n}$ is said to be \emph{convex} if $x<z<y$ implies $z\in A$ for all $x,y\in A$ and $z\in X_{n}$. For any two non-empty subsets $A$ and $B$ of $X_{n}$, we write $A<B$ if $a<b$ for all $a\in A$ and $b\in B$. Let $P=\{ A_{1}, \ldots ,A_{r} \}$ ($1\leq r\leq n$) be a partition of $X_{n}$, that is a family of non-empty disjoint subsets of $X_{n}$ whose union is $X_{n}$. If $A_{i}$ is a convex subset of $X_{n}$ for each $1\leq i\leq r$, then $P$ is called  a \emph{convex} partition of $X_{n}$, and if $A_{i}< A_{i+1}$ for every $1\leq i\leq r-1$, then $P$ is called an \emph{ordered} partition of $X_{n}$. A transformation $\alpha\in \mathcal{O}_{n}$ with $\im(\alpha)=\{a_{1}<\cdots<a_{r}\}$ can be expressed as
$\alpha=\left(
\begin{matrix}
A_{1}  & \cdots & A_{r} \\
a_{1}  & \cdots & a_{r}
\end{matrix}\right)$ 
where $A_{i}=a_{i}\alpha^{-1}$ for every $1\leq i\leq r$. Then notice that  $\{A_{1}, \ldots, A_{r}\}$ is a convex and ordered partition of $X_{n}$. 

Let $S$ be any subsemigroup of $\mathcal{O}_{n}$, and for each $1\leq r\leq n$, we let 
$$D_{r}= D_{r}(S) =\{ \alpha\in S : \lvert \im(\alpha) \rvert =r\} .$$
Moreover, for each $2\leq i\leq n-1$, we define
\begin{align*}
\xi_{i}&=\left(\begin{array}{ccccccccccc}
1  & \cdots & i-2 & i-1 & i & \cdots & n-1 &  n \\
1  & \cdots & i-2 &  i  & i & \cdots & n-1 & n-1
\end{array}\right) \, \mbox{ and} \\ 
\zeta_{i}&=\left(\begin{array}{ccccccccccc}
1  & 2 & \cdots & i & i+1 & i+2 & \cdots & n \\
2  & 2 & \cdots & i &  i  & i+2 & \cdots & n
\end{array}\right),  
\end{align*}
and we let
$$E^{+}=\{\xi_{i} : 2\leq i\leq n-1\} \mbox{ and } E^{-}=\{ \zeta_{i}: 2\leq j\leq n-1\}.$$ 
Then it is clear that $E^{+} \subseteq \mathsf{L}_{k}$ for each $1\leq k\leq n-1$, and that $E^{-} \subseteq \mathsf{L}_{k}$ for each $2\leq k\leq n$. Furthermore, for  each $1\leq i\leq n-1$, we let 
\begin{eqnarray*}
	&&\beta_{i}=\left(\begin{array}{ccccccc}
		1 & 2 &\cdots &  i  & i+1 &\cdots & n \\
		2 & 3 &\cdots & i+1 & i+1 &\cdots & n
	\end{array}\right) \mbox { and} \\
	&&\gamma_{i}=\left(\begin{array}{ccccccc}
		1 &\cdots & i & i+1 & i+2 &\cdots  &  n \\
		1 &\cdots & i &  i  & i+1 &\cdots  & n-1
	\end{array}\right),
\end{eqnarray*} 
then it is also clear that 
\begin{eqnarray*}
	D_{n-1}(\mathsf{L}_{1}) =\{ \gamma_{i} : 1\leq i\leq n-1\} &\mbox{and}& D_{n-1}(\mathsf{L}_{n}) =\{ \beta_{i} : 1\leq i\leq n-1\}.
\end{eqnarray*} 
Remember that $\mathcal{IO}_{n}$ denotes the subsemigroup of $\mathcal{O}_{n}$ consisting of all transformations in $\mathcal{O}_{n}$ whose images are intervals (convex subsets) of $X_{n}$. It is shown in \cite{Fernandes&Paulista:2023} that $D_{n-1}(\mathsf{L}_{1}) \cup D_{n-1}(\mathsf{L}_{n})$ is a generating set of $\mathcal{IO}_{n}$. More importantly, we have the following result from \cite{Fernandes&Paulista:2023}.

\smallskip

\begin{theorem}\label{t4}\cite[Theorem 3.5]{Fernandes&Paulista:2023}
	For $n\geq 3$, $\{ \gamma_{1} ,\ldots, \gamma_{n-2}, \beta_{n-1} \}$ is a minimal generating set of $\mathcal{IO}_{n}$, and so $\rank(\mathcal{IO}_{n})=n-1$. \hfill $\square$
\end{theorem}

\smallskip

Furthermore, for a non-empty subset $Y$ of $X_{n}$, let $\mathcal{O}_{n}(Y) =\{ \alpha\in \mathcal{O}_{n}: \im(\alpha)\subseteq Y\}$ which is a subsemigroup of $\mathcal{O}_{n}$ and studied in \cite{Fernandes&Honyam&. Quinteiro& Singha:2014}. An element $y\in Y$ is called \emph{captive} if either $y\in \{1,n\}$ or $1<y<n$ and $y-1,\, y+1\in Y$. Now let $Y^{\sharp}$ denote the subset of captive elements of $Y$. With this notation, Fernandes et. al. proved the following theorem.

\begin{theorem}\label{t5}\cite[Theorem 4.3]{Fernandes&Honyam&. Quinteiro& Singha:2014}
	Let $1<r<n$ and $Y$ be a subset of $X_{n}$ with $r$ elements. Then $\rank(\mathcal{O}_{n}(Y))=\binom{n-1}{r-1}+\lvert Y^{\sharp}\rvert$. \hfill $\square$
\end{theorem}

By Theorem \ref{t5}, since $\mathsf{L}_{1}=\mathcal{O}_{n}(X_{n-1})$ and $X_{n-1}^{\sharp} =X_{n-2}$, and by the duality, we have the following immediate corollary.

\smallskip

\begin{corollary}\label{c6}
	$\rank(\mathsf{L}_{1})=\rank(\mathsf{L}_{n})=2n-3$. \hfill $\square$
\end{corollary}

From now on $\mathsf{L}_{k}$ for $2\leq k\leq n-1$. Since $\mathsf{L}_{2} =\cdots= \mathsf{L}_{n-1}$ for $n\geq 3$, we fix $k=2$ and consider the semigroup $\mathsf{L}_{2}$. For $n=3$, it is also easy to check that
$\mathsf{L}_{2} =\langle \left\{ \left( \begin{smallmatrix}
1 & 2 & 3 \\
1 & 1 & 2
\end{smallmatrix}\right),  \left( \begin{smallmatrix}
1 & 2 & 3 \\
2 & 3 & 3
\end{smallmatrix}\right)\right\}  \rangle$, and that $\mathsf{L}_{2}$ has rank $2$. So we consider the case $n\geq 4$. Moreover, it is shown in \cite[Lemmas 4.9 and 4.10]{Fernandes&Honyam&. Quinteiro& Singha:2014} that $D_{n-1}(\mathsf{L}_{1}) \cup E^{+}$ and  $D_{n-1}(\mathsf{L}_{n}) \cup E^{-}$ are minimal generating sets of $\mathsf{L}_{1}$ and $\mathsf{L}_{n}$, respectively. Thus, $D_{n-1}(\mathsf{L}_{1}) \cup E^{+}\cup D_{n-1}(\mathsf{L}_{n}) \cup E^{-}$ is a generating set of $\mathsf{L}_{2} =\mathsf{L}_{1} \cup \mathsf{L}_{n}$. By using these results, we have the following.

\smallskip

\begin{lemma}\label{l7}
	For $n\geq 4$, $B=\{\gamma_{1}, \beta_{2}, \dots, \beta_{n-1}, \xi_{3}, \dots, \xi_{n-1}\}$ is a generating set of $\mathsf{L}_{2}$. 	
\end{lemma}

\begin{proof}
In order to prove that $B$ is a generating set of $\mathsf{L}_{2}$, we need to show that $\beta_{1}, \xi_{2}\in \langle B\rangle$ and $(D_{n-1}(\mathsf{L}_{1}) \setminus \{\gamma_{1}\}) \cup E^{-}\subseteq  \langle B\rangle$. It is routine matter to check that $\beta_{1} =\gamma_{1} \beta_{n-1}$,\, $\gamma_{i} =\beta_{i} \gamma_{1}$ for each $2\leq i\leq n-1$ and $\zeta_{n-1} =\xi_{2}= \beta_{1} \beta_{n-1} \gamma_{1}$, and so $\mathsf{L}_{1}= \langle D_{n-1}(\mathsf{L}_{1}) \cup E^{+} \rangle \subseteq \langle B\rangle$. For each $2\leq i\leq n-2$, if we let 
	$$\zeta^{'}_{i} =\left( \begin{matrix}
	1  & 2 & \cdots &  i  & i+1 & i+2 & \cdots &  n \\
	1  & 1 & \cdots & i-1 & i-1 & i+1 & \cdots & n-1
	\end{matrix}\right),$$
then it is clear that $\zeta^{'}_{i}\in \mathsf{L}_{1}$, and so $\zeta^{'}_{i}\in \langle B\rangle$. Since $\zeta_{i}= \zeta^{'}_{i} \beta_{n-1}$ for each $2\leq i\leq n-2$, the proof is now completed.
\end{proof}

\smallskip

\begin{theorem}\label{t8}
	For  $n\geq 3$, the rank of $\mathsf{L}_{2}$ is $2n-4$.
\end{theorem}

\begin{proof}
We previously observed that  $\rank(\mathsf{L}_{2})=2$ for $n=3$. Now, we consider the case  $n\geq 4$. Let $A$ be any generating set of $\mathsf{L}_{2}$, and let $D_{r}=D_{r}(\mathsf{L}_{2})$ for each $1\leq r\leq n-1$. Observe that $D_{n-1}= D_{n-1}(\mathcal{IO}_{n})$. By Theorem \ref{t4}, since $\{ \gamma_{1} ,\ldots, \gamma_{n-2}, \beta_{n-1} \} \subseteq D_{n-1}$ is a minimal generating set of $\mathcal{IO}_{n}$, and so of $D_{n-1}$, it follows that $\langle D_{n-1}\rangle =\mathcal{IO}_{n}\neq \mathsf{L}_{2}$, and that $A$ contains at least $n-1$ elements from $D_{n-1}$. As a consequence of this fact, we make the assumption that $A\cap D_{n-2}(\mathcal{IO}_{n}) =\emptyset$.  Since $\zeta_{n-1},\xi_{2}\in \mathcal{IO}_{n}$,  we consider the sets: $U= \bigcup\limits_{i=2}^{n-2} U_{i}$ and $V= \bigcup\limits_{i=2}^{n-2} V_{i}$ where 
\begin{eqnarray*}
	U_{i}&=& \{\alpha\in \mathsf{L}_{2}:\im(\alpha)=X_{n}\setminus \{i,n\}\}\,  \mbox{ and} \\ 
	V_{i}&=& \{\alpha\in L_{2}: \im(\alpha)=X_{n}\setminus \{1,i+1\}\}  
\end{eqnarray*}
for $2\leq i\leq n-2$. Then, it is routine matter to check that $D_{n-2}$ is a disjoint union of $U$, $V$ and $D_{n-2}(\mathcal{IO}_{n})$. Next we show that $(U_{i}\cup V_{i})\cap A\neq \emptyset$ for each $2\leq i\leq n-2$.  
	
Let $2\leq i\leq n-2$,\, $\alpha \in U_{i}$ and $\delta \in V_{i}$. Then it is clear that $\alpha \beta_{n-1} \in V_{i}$ and $\delta \gamma_{1}\in U_{i}$. Since $\alpha= \alpha \beta_{n-1} \gamma_{1}$ and $\delta= \delta \gamma_{1} \beta_{n-1}$, without loss of generality, we suppose that $A\cap U_{i}=\emptyset$ for all $2\leq i\leq n-2$. For each $2\leq i\leq n-2$, if we consider
\begin{eqnarray*}
&&\alpha =\left(\begin{matrix}
	1   & \cdots & i-2 & i-1 & i   & \cdots & n-3 & n-2 & n-1 & n \\
	2   & \cdots & i-1 & i & i+2 & \cdots & n-1 & n   & n   & n
\end{matrix}\right)\in V_{i},
\end{eqnarray*}
then there exist $\alpha_{1} ,\ldots, \alpha_{t}\in A$ such that $\alpha= \alpha_{1} \cdots \alpha_{t}$. Since $\rvert \im(\alpha) \lvert =n-2$, it follows that $n-2\leq \rvert \im(\alpha_{i}) \lvert \leq n-1$ for each $1\leq i\leq t$, and since we suppose that $A\cap U=A\cap D_{n-2}( \mathcal{IO}_{n}) =\emptyset$, we conclude that $\alpha_{1}, \ldots, \alpha_{t} \in D_{n-1} \cup V$. If $\alpha_{t} \in V$, then we immediately conclude that $\im(\alpha)= \im(\alpha_{t}) \in V_{i}$ since $\im(\alpha) \subseteq \im(\alpha_{t})$ and $\lvert \im(\alpha_{t}) \rvert =n-2$.    
	
Suppose $\alpha_{t} \in D_{n-1}$. Then $s=\max\{ \,i: \alpha_{i}\in V \}$ exists and $1\leq s\leq t-1$ since $\langle D_{n-1}\rangle =\mathcal{IO}_{n}\neq \mathsf{L}_{2}$. Let $\beta= \alpha_{1} \cdots \alpha_{s}$ and $\gamma= \alpha_{s+1} \cdots \alpha_{t}$ so that $\alpha =\beta \gamma$. Since $\gamma \in \mathcal{IO}_{n}$ and $2,n\in \im(\alpha) \subseteq \im(\gamma)$, it follows that $\im(\gamma) =\{2, \ldots, n\}$. Since $1\notin \im(\alpha_{s})$ and $\im(\beta) \subseteq \im(\alpha_{s})$, we notice that $2\leq 1\beta$, and so $2=1\alpha =(1\beta) \gamma \geq 2\gamma \geq 2$, i.e $1\gamma =2\gamma =2$. Thus, we have $\gamma_{ \mid_{\{2, \ldots, n\}}} =1_{\mid_{\{2, \ldots, n\}}}$, and so since $\im(\beta) \subseteq \{2, \ldots, n\}$ we conclude that $\alpha =\beta \gamma =\beta$, and so $\im(\alpha) = \im(\alpha_{s}) \in V_{i}$, as required. 
\end{proof}

Consider the semigroups $\mathsf{R}_{1}$ and $\mathsf{R}_{n}$, and consider the mapping $\varphi : \mathsf{R}_{1} \rightarrow \mathsf{R}_{n}$ which maps each transformation $\alpha\in \mathsf{R}_{1}$ into the transformation $\widehat{\alpha} \in \mathsf{R}_{n}$ defined by $x\widehat{\alpha} =n+1-(n+1-x)\alpha$ for all $x\in X_{n}$. It is a routine matter to check that $\varphi$ is an isomorphism. Consequently, $\mathsf{R}_{1}$ and $\mathsf{R}_{n}$ are isomorphic subsemigroups of $\mathcal{O}_{n}$. Thus, we only consider $\mathsf{R}_{1}=\{\alpha\in\mathcal{O}_{n}: 1\alpha = 2\alpha =1 \}$. It is also a routine matter to check that 
$\mathsf{R}_{1}= \left\{ \left( \begin{smallmatrix}
	1 & 2\\
	1 & 1
\end{smallmatrix}\right)\right\}$ for $n=2$ and
$\mathsf{R}_{1} =\left\{ \left( \begin{smallmatrix}
	1 & 2 & 3 \\
	1 & 1 & 1
\end{smallmatrix}\right),  \left( \begin{smallmatrix}
	1 & 2 & 3 \\
	1 & 1 & 2
\end{smallmatrix}\right), \left( \begin{smallmatrix}
	1 & 2 & 3 \\
	1 & 1 & 3
\end{smallmatrix}\right)\right\}$, whence the rank of $\mathsf{R}_{1}$ is $1$ for $n=2$, and is $2$ for $n=3$. Suppose $n\geq 4$ and let
\begin{eqnarray*}
&&\delta_{3}=\left(\begin{matrix}
	1  & 2 & 3 & 4 &\cdots & n \\
	1  & 1 & 4 & 4 &\cdots & n
\end{matrix}\right), \\
&&\delta_{i}=\left(\begin{matrix}
	1  & 2 & 3 & \cdots & i-1 &  i & i+1 & \cdots & n \\
	1  & 1 & 3 & \cdots & i-1 & i+1 & i+1 & \cdots & n
\end{matrix}\right) \,  \mbox{ for }\,  4\leq i\leq n-1, \mbox{ and}\\
&&\lambda_{i}=\left(\begin{array}{ccccccc}
	1  & 2 &\cdots &  i  & i+1 &\cdots & n \\
	1  & 1 &\cdots & i-1 & i+1 &\cdots & n
\end{array}\right) \, \mbox{ for }\,  2\leq i\leq n.
\end{eqnarray*} 
Moreover, we define the following two sets:
\begin{eqnarray*}
	C=\{ \delta_{i} : 3\leq i\leq n-1\},  &&\,\,\, F=\{ \lambda_{i} : 2\leq i\leq n\}, \\
	\mathsf{R}_{1}^{*}=\{\alpha\in \mathsf{R}_{1}:3\alpha \geq 3\}\quad &\mbox{ and }\,\, &D_{r}= D_{r}(\mathsf{R}_{1}) 
\end{eqnarray*}
for $2\leq r\leq n-1$. Then it is clear that $F=D_{n-1}(\mathsf{R}_{1})$. For each $1\leq i\leq n-1$, let $\theta_{i}$ be the idempotent of $\mathcal{O}_{n}$ such that $\fix(\theta_{i}) =X_{n} \setminus \{i\}$ and $i\theta_{i} =i+1$, and let 
\begin{eqnarray}\label{e1}
G_{n}=\{\theta_{1}, \ldots, \theta_{n-1}, \lambda_{n}, 1_{n}\}.
\end{eqnarray} 
It is shown in \cite[Theorem 2.7]{Gomes&Howie:1992} that $G_{n}$ is a minimal generating set of $\mathcal{O}_{n}$ for each $n\geq 2$. Now we have the following result.

\smallskip

\begin{proposition} \label{p9}
	$\mathcal{O}_{n-2}$ and $\mathsf{R}_{1}^{*}$ are isomorphic semigroups, and hence, $C\cup \{\lambda_{2}, \delta_{3} \lambda_{n}\}$ is a minimal generating set of $\mathsf{R}_{1}^{*}$ for $n\geq 4$.
\end{proposition}

\begin{proof}
First notice that 
$\delta_{3} \lambda_{n} =\left( \begin{matrix}
	1  & 2 & 3 & 4 & 5 &  \cdots & n \\
	1  & 1 & 3 & 3 & 4 &  \cdots & n-1
\end{matrix}\right) \in \mathsf{R}_{1}^{*}$, 
and so $C\cup \{\lambda_{2}, \delta_{3} \lambda_{n}\}$ is a subset of $\mathsf{R}_{1}^{*}$.
	
Let $\alpha$ be an element in $\mathsf{R}_{1}^{*}$. Then consider the map  $\check{\alpha} :X_{n-2} \rightarrow X_{n-2}$ defined by $x\check{\alpha} =(x+2)\alpha -2$ for all  $x\in  X_{n-2}$, and the map $\phi: \mathsf{R}_{1}^{*} \rightarrow \mathcal{O}_{n-2}$ defined by $\alpha \phi= \check{\alpha}$ for all $\alpha \in \mathsf{R}_{1}^{*}$. It is easy to check that $\check{\alpha} \in \mathcal{O}_{n-2}$, and that $\phi$ is an isomorphism. Moreover, since $(E\cup \{\lambda_{2}, \delta_{3} \lambda_{n}\}) \phi= G_{n-2}$, as defined in (\ref{e1}), is a minimal generating set of $\mathcal{O}_{n-2}$, it follows from the isomorphism that $C\cup \{\lambda_{2}, \delta_{3} \lambda_{n}\}$ a minimal generating set of $\mathsf{R}_{1}^{*}$.
\end{proof}

\smallskip

\begin{lemma}\label{l10}
	For $n\geq 4$, $C\cup F$ is a generating set of $\mathsf{R}_{1}$.
\end{lemma}

\begin{proof} 
By Proposition \ref{p9}, it is enough to show that $\mathsf{R}_{1} \setminus \mathsf{R}_{1}^{*} \subseteq \langle C\cup F \rangle$. For $\alpha \in \mathsf{R}_{1} \setminus \mathsf{R}_{1}^{*}$, if $\lvert \im(\alpha)\rvert=n-1$, then we immediately have $\alpha\in F$. So we suppose that $\alpha \in \mathsf{R}_{1} \setminus \mathsf{R}_{1}^{*}$ with size at most $n-2$. Then $\alpha$ has the tabular form: 
$\alpha=\left( \begin{matrix}
	A_{1} & A_{2} &\cdots & A_{r} \\
	1   & a_{2} &\cdots & a_{r}
\end{matrix}\right) \in D_{r}$ 
with $1=a_{1}< a_{2}<\cdots< a_{r}$ ($2\leq r\leq n-2$). Note that $2\in A_{1}$ and there exists $2\leq j\leq n$ such that $\{1,\ldots ,j-1\} \subseteq \im(\alpha)$ but $j\notin \im(\alpha)$, and so $a_{j}\geq j+1$. Then we consider two cases; either $3\notin A_{1}$, i.e. $A_{1}=\{1,2\}$ or $3\in A_{1}$.
	
\emph{Case} $1$. Suppose that $A_{1}=\{1,2\}$. Then note that $3\in A_{2}$, and so $2=a_{2}=3\alpha$ since $1< 3\alpha <3$. Thus, $j\geq 3$. If we let 
\begin{eqnarray}\label{e2}
\beta=\left( \begin{matrix}
	A_{1} & A_{2} &\cdots & A_{j-1} & A_{j} & A_{j+1} &\cdots & A_{r} \\
	1   &   3   &\cdots &    j    & a_{j} & a_{j+1} &\cdots & a_{r}
\end{matrix}\right) \in D_{r},
\end{eqnarray} 
then it is clear that $\beta\in \mathsf{R}_{1}^{*}$. Since $\alpha =\beta \lambda_{j}$, it follows from Proposition \ref{p9} that $\alpha \in \langle C\cup F \rangle$.
	
\emph{Case} $2$. Suppose that $3\in A_{1}$. Then we note that $B_{1}= A_{1}\setminus \{1,2\} \neq \emptyset$. If we similarly let 
\begin{eqnarray}\label{e3}
\beta=\left( \begin{matrix}
	\{1,2\}&B_{1}&A_{2}&\cdots& A_{j-1} & A_{j} & A_{j+1} &\cdots & A_{r} \\
	1   &  2  &  3  &\cdots&    j    & a_{j} & a_{j+1} &\cdots & a_{r}
\end{matrix}\right) \in D_{r+1},
\end{eqnarray}
then it is also clear that $\beta\in \mathsf{R}_{1} \setminus \mathsf{R}_{1}^{*}$ with $3\beta =2$. Since $\alpha =\beta \lambda_{j}$, it follows from the first case and Proposition \ref{p9} that $\alpha \in \langle C\cup F \rangle$.  Therefore, we conclude that $\langle C\cup F\rangle= \mathsf{R}_{1}$.
\end{proof}

For any subset $\pi$ of $X_{n}\times X_{n}$, let $\pi^{e}$ denote the smallest equivalence relation on $X_{n}$ containing $\pi$. Then we have one of the main theorems of this paper.

\smallskip

\begin{theorem}\label{t11}
	For  $n\geq 3$, $\rank(\mathsf{R}_{1})=\rank(\mathsf{R}_{n})=2n-4$.
\end{theorem}	

\begin{proof}
First note that $D_{n-1}=F$ and that for each $2\leq j\leq n$ there exists a unique $\alpha \in D_{n-1}$ such that $\im(\alpha) =X_{n}\setminus \{ j\}$, namely $\alpha =\lambda_{j}$. For any $\alpha \in F$, assume that $\alpha = \beta\gamma$ for some $\beta, \gamma\in \mathsf{R}_{1}$. Since $\im(\alpha) \subseteq \im(\gamma)$ and $n-1=\lvert \im(\alpha) \rvert \leq\lvert \im(\gamma) \rvert\leq n-1$, it follows that $\im(\alpha) =\im(\gamma)$, and so $\alpha =\gamma$. Thus, each element of $F$ is undecomposable element in $\mathsf{R}_{1}$. 
	
Let $A$ be any generating set of $\mathsf{R}_{1}$, let $3\leq i\leq n-1$, and suppose that $\delta_{i}= \beta \gamma$ for some $\beta, \gamma\in \mathsf{R}_{1}$. Since $(1,2)\in \ker(\alpha)$ for all $\alpha \in \mathsf{R}_{1}$, and since $\ker(\beta) \subseteq \ker(\delta_{i}) =\{(1,2), (i,i+1) \}^{e}$, we have either $\ker(\beta) =\{(1,2) \}^{e}$ or $\ker(\beta) =\ker(\delta_{i})$. In the first case, since $\lvert \im(\beta)\rvert =n-1$, we have $\beta =\lambda_{j}$ for some $2\leq j\leq n$. For all $j\neq 2$, since $3\lambda_{j} \gamma= 2\gamma=1\neq 3\delta_{i}$, we must have $\beta =\lambda_{2}$, and so $\delta_{i}= \gamma$. Thus, in the both cases, we have an element $\beta\in A\setminus F$ such that $\ker(\beta) =\ker(\delta_{i})$, and so $\lvert A\lvert \geq 2n-4$. Therefore, by Lemma \ref{l10}, we have $\rank(\mathsf{R}_{1})=2n-4$, as claimed.  
\end{proof}

Finally, we focus on the semigroups $\mathsf{Z}_{1}$ and $\mathsf{Z}_{n}$. If we consider the isomorphism $\varphi: \mathsf{R}_{1} \rightarrow \mathsf{R}_{n}$, which is defined after the proof of Theorem \ref{t8}, then it is also a routine matter to check that the restriction of $\varphi$ to $\mathsf{Z}_{1}$ is an isomorphism from $\mathsf{Z}_{1}$ onto $\mathsf{Z}_{n}$. Thus, we only consider $\mathsf{Z}_{1} =\{\alpha\in \mathcal{O}_{n}: 1\alpha = 2\alpha =1\mbox{ and }  n\alpha \leq n-1\}$. It is clear that $\mathsf{Z}_{1}= \left\langle \left( \begin{smallmatrix}
1 & 2 & 3\\
1 & 1 & 2
\end{smallmatrix}\right)\right\rangle$ for $n=3$ and that  
$\left\{ \left( \begin{smallmatrix}
1 & 2 & 3  & 4 \\
1 & 1 & 2  & 3
\end{smallmatrix}\right),  \left( \begin{smallmatrix}
1 & 2 & 3 & 4\\
1 & 1 & 3 & 3
\end{smallmatrix}\right)\right\}$ is a minimal generating set of $\mathsf{Z}_{1}$ for $n=4$,  and so the rank of $\mathsf{Z}_{1}$ is $1$ for $n=3$, and is $2$ for $n=4$. Suppose $n\geq 5$ and let
\begin{eqnarray*}
&&\mu_{i}=\left( \begin{matrix}
	1  & 2 & 3 & \cdots  & i & i+1 & \cdots  & n \\
	1  & 1 & 3 & \cdots  & i & i   & \cdots  & n-1
	\end{matrix}\right) \, \mbox{ for }\, 3\leq i\leq n-1,\\
&&\rho_{3}=\left( \begin{matrix}
	1  & 2 & 3 & 4 & 5 & \cdots & n-1  & n \\
	1  & 1 & 4 & 4 & 5 & \cdots & n-1  & n-1
\end{matrix}\right), \\
&&\rho_{i}=\left( \begin{matrix}
		1  & 2 & 3 & \cdots & i-1 & i   & i+1 & \cdots & n-1  & n \\
		1  & 1 & 3 & \cdots & i-1 & i+1 & i+1 & \cdots & n-1  & n-1
	\end{matrix}\right) \, \mbox{ for }\, 4\leq i\leq n-2,  \, \mbox{ and}\\
&&\tau_{i}=\left(\begin{matrix}
		1  & 2 & 3 & \cdots & i   & i+1 & i+2 & \cdots & n-1  & n \\
		1  & 1 & 2 & \cdots & i-1 & i+1 & i+2 & \cdots & n-1  & n-1
	\end{matrix}\right) \, \mbox{ for }\, 3\leq i\leq n-1.
\end{eqnarray*}
Notice that $\tau_{n-1}=\gamma_{1}$. Moreover, we define the following sets:
\begin{eqnarray*}
	H=\{\mu_{i} : 3\leq i\leq n-1\}, && K=\{\rho_i: 3\leq i\leq n-2\},  \\
	M=\{\tau_{i}: 3\leq i\leq n-1\} &\mbox{and}&
	\mathsf{Z}_{1}^{*}=\{\alpha\in \mathsf{Z}_{1}: 3\alpha\geq 3\}. 
\end{eqnarray*}
First observe that $D_{n-1}(\mathsf{Z}_1)=\{\tau_{n-1}\}$ and $D_{n-2}(\mathsf{Z}_{1}^{*}) =H$. Since we now consider $\mathsf{L}_{1}$ and $E^{+}$ for different $n$'s, we use the notations $\mathsf{L}_{1,n}$ and $E_{n}^{+}$ for $\mathsf{L}_{1}$ and $E^{+}$, respectively. For every $\alpha \in \mathsf{Z}_{1}^{*}$, let $\check{\alpha} :X_{n-2} \rightarrow X_{n-2}$ defined as in the proof of Proposition \ref{p9}. If we define $\phi: \mathsf{Z}_{1}^{*} \rightarrow \mathsf{L}_{1,n-2}$ by $\alpha \phi=\check{\alpha}$ for all $\alpha \in \mathsf{Z}_{1}^{*}$, then it is also a routine matter to check that  $\phi$ is an isomorphism. By Corollary \ref{c6}, since $(H\cup K)\phi= E_{n-2}^{+} \cup D_{n-3}(L_{1,n-2})$ is a minimal generating set of $\mathsf{L}_{1,n-2}$, and so we have the following immediate corollary.

\smallskip

\begin{corollary} \label{c12}
	For $n\geq 5$,	$H\cup K$ is a minimal generating set of $\mathsf{Z}_{1}^{*}$. \hfill $\square$
\end{corollary}

Then we have the following.

\smallskip

\begin{lemma} \label{l13}
	For $n\geq 5$, $H\cup K\cup M$ is a generating set of $\mathsf{Z}_{1}$. 
\end{lemma}

\begin{proof}
By Corollary \ref{c12}, it is enough to show that $\mathsf{Z}_{1}\setminus \mathsf{Z}_{1}^{*}  \subseteq \langle H\cup K\cup M\rangle$. For $\alpha \in \mathsf{Z}_{1} \setminus \mathsf{Z}_{1}^{*}$, if $\lvert \im(\alpha) \rvert =n-1$, then we immediately have $\alpha= \tau_{n-1}\in M$. So we suppose that $\alpha \in \mathsf{Z}_{1} \setminus \mathsf{Z}_{1}^{*}$ with size at most $n-2$. Then $\alpha$ has the tabular form: 
$\alpha=\left( \begin{matrix}
	A_{1} & A_{2} &\cdots & A_{r} \\
	1   & a_{2} &\cdots & a_{r}
\end{matrix}\right) \in D_{r}$ 
with $1=a_{1}< a_{2}<\cdots< a_{r}$ ($2\leq r\leq n-2$). Note that $2\in A_{1}$ and there exists $2\leq j\leq n-1$ such that $\{1,\ldots ,j-1\} \subseteq \im(\alpha)$ but $j\notin \im(\alpha)$, and so $a_{j}\geq j+1$. Then we consider two cases as above; either $3\notin A_{1}$ or $3\in A_{1}$.
	
\emph{Case} $1$. Suppose that $A_{1}=\{1,2\}$. Then note that $3\in A_{2}$, and similarly, $a_{2}=2$ and $j\geq 3$. Similarly, we consider $\beta$ as defined in (\ref{e2}). Since $\beta\in \mathsf{Z}_{1}^{*}$ and $\alpha =\beta \tau_{j}$, it follows from Corollary \ref{c12} that $\alpha \in \langle H\cup K\cup M \rangle$.
	
\emph{Case} $2$. Suppose that $3\in A_{1}$. Similarly, we consider $\beta$ as defined in (\ref{e3}). Since $\beta\in \mathsf{Z}_{1} \setminus \mathsf{Z}_{1}^{*}$ with $3\beta =2$, and since $\alpha =\beta \tau_{j}$, it follows from the first case and Corollary \ref{c12} that $\alpha \in \langle H\cup K\cup M \rangle$. Therefore, we have $\langle H\cup K\cup M \rangle=\mathsf{Z}_{1}$.
\end{proof}

\smallskip

\begin{theorem}\label{t14}
	For $n\geq 5$, $\rank(\mathsf{Z}_{1})=\rank(\mathsf{Z}_{n})=2n-5$.	
\end{theorem}

\begin{proof}
Let $A$ be any  generating set of $\mathsf{Z}_{1}$. Since $D_{n-1}(\mathsf{Z}_{1})=\{\tau_{n-1}\}$, $\tau_{n-1}$ is clearly undecomposable. Let $3\leq i\leq n-2$ and assume that $\tau_{i} =\alpha \beta$ for some $\alpha ,\beta\in \mathsf{Z}_{1}$. Since $\tau_{i}$ is injective on $\{3,\dots,n-1\}$,  $\alpha$ is also injective on $\{3,\dots,n-1\}$. Moreover, since $1\leq 3\alpha \leq 2$ gives the contradiction $2=3\tau_{i}=3\alpha \beta =1$, it follows that $\alpha$ is a bijection from $\{3,\dots,n-1\}$ onto itself, and so $x\beta=x$ for every $3\leq x\leq n-1$. Thus, $\beta=\tau_{i}$, i.e. each element of $M$ is undecomposable element in $\mathsf{Z}_{1}$. Thus, $A$ contains $M$. 

Let $3\leq i\leq n-1$ and suppose that $\mu_{i}= \beta \gamma$ for some $\beta, \gamma\in \mathsf{Z}_{1}$. Similarly, either $\ker(\beta) =\{(1,2) \}^{e}$ or $\ker(\beta) =\ker(\mu_{i})$. In the first case, since $\lvert \im(\beta) \rvert =n-1$, we have $\beta =\tau_{n-1}$, and moreover, since $3\tau_{n-1}=2$, we obtain $3\mu_{i}=2$ which is a contradiction. Thus, we must have $\ker(\beta) =\ker(\mu_{i})$. For all $3\leq i\leq n-2$ and $3\leq j\leq n-1$, notice that $\ker(\mu_{i}) \neq\ker(\tau_{j})$,  but for all $3\leq j\leq n-2$, $\ker(\mu_{n-1}) =\ker(\tau_{j})$. Since $3\tau_{j} \gamma =1$ for all $3\leq j\leq n-1$ and $\gamma\in \mathsf{Z}_{1}$, we also notice that $\beta \neq \tau_{j}$ for all $3\leq j\leq n-1$,  and so $A$ must contain at least $n-3$ elements from $D_{n-2}(\mathsf{Z}_{1})$. Since every element $\alpha$ in $D_{n-2}(\mathsf{Z}_{1}) \cup \{\tau_{n-1}\}$ is order-decreasing but $\mathsf{Z}_{1}$ contains order-increasing elements, we conclude that $\langle D_{n-2}(\mathsf{Z}_{1})\cup \{\tau_{n-1}\} \rangle \neq \mathsf{Z}_{1}$, and so $\lvert A\rvert >(n-3)+(n-3)=2n-6$. Finally, consider 
$$\rho=\left(\begin{matrix}
	1  & 2 & 3 & 4 &\cdots & n-2 & n-1  & n \\
	1  & 1 & 4 & 5 &\cdots & n-1 & n-1 & n-1 
\end{matrix}\right)\in D_{n-3}(Z_{1}),$$ 
and observe that $\mu_{i}\rho\tau_{i}=\rho_{i}$ for each $3\leq i\leq n-2$. Therefore, it follows from Lemma \ref{l13} that $H\cup M\cup\{\rho\}$ is a minimal generating set of $\mathsf{Z}_{1}$, and so we have $\rank(\mathsf{Z}_{1})=\rank(\mathsf{Z}_{n})=2n-5$.
\end{proof}

Notice that if we consider
$$\alpha =\left(\begin{matrix}
	1  & 2 & 3 & 4 &\cdots & n-2 & n-1  & n \\
	1  & 1 & 3 & 5 &\cdots & n-1 & n-1 & n-1 
\end{matrix}\right)\, \mbox{ and }\, 
\beta =\left(\begin{matrix}
	1  & 2 & 3 & 4 & 5 &\cdots & n-1  & n \\
	1  & 1 & 4 & 4 & 5 &\cdots & n-1 & n-1 
\end{matrix}\right),$$ 
then it is clear that $\alpha, \beta\in D_{n-3}(Z_{1})$, and that $\rho= \alpha \beta$, and so the mapping $\rho$ defined in the above proof is not undecomposable in $Z_{1}$.


\section*{ORCID}

\noindent Emrah Korkmaz - \url{https://orcid.org/0000-0002-4085-0419}

\noindent Hayrullah Ayık- \url{https://orcid.org/0000-0001-7569-5597}

\end{document}